\documentclass[12pt, a4paper]{article}
\usepackage{graphicx}
\usepackage{amssymb}
\usepackage{stmaryrd}
\usepackage{enumitem}

\usepackage{geometry}
\usepackage{tikz}
\usepackage{float}
\usepackage{url}
\usepackage{mathtools}
\usepackage{amsmath,amsfonts,amsthm}
\usepackage{mathrsfs}

\theoremstyle{plain}  
\newtheorem{theorem}{Theorem}[section]
\newtheorem{lemma}[theorem]{Lemma}
\newtheorem{proposition}{Proposition}[section]

\theoremstyle{definition}  
\newtheorem{definition}{Definition}

\newtheorem{problem}{Problem}

\usepackage{pgfplots}
\pgfplotsset{compat=1.18}
\newcommand{\GG}{\mathcal{G}}
\newcommand{\TT}{\mathcal{T}}
\newcommand{\TS}{\mathcal{S}}
\newcommand{\CC}{\mathbb{C}}

\usepackage{subfig}

\usepackage{ytableau}
\usepackage[colorlinks,
linkcolor=blue,
anchorcolor=blue,
citecolor=blue
]{hyperref}
\begin{document}
\begin{center}
{\large \bf  On Extremal Family Trees $(\mathcal{T}_n)_{n\geqslant 3}$ Beyond Caterpillars and Greedy Constructions}
\end{center}
\begin{center}
 Jasem Hamoud$^1$ \hspace{0.2 cm}  Duaa Abdullah$^{2}$\\[6pt]
 $^{1,2}$ Physics and Technology School of Applied Mathematics and Informatics \\
Moscow Institute of Physics and Technology, 141701, Moscow region, Russia\\[6pt]
Email: 	 $^{1}${\tt jasem1994hamoud@gmail.com},
 $^{2}${\tt abdulla.d@phystech.edu}
\end{center}
\noindent
\begin{abstract}
This paper investigates topological indices for the greedy tree $\TT_\mathscr{D}$ associated with a graphic degree sequence $\mathscr{D} = (d_1 \geqslant d_2 \geqslant \dots \geqslant d_n)$ of a tree. A fundamental challenge in the study of topological indices lies in establishing precise bounds, as such findings illuminate intrinsic relationships among diverse indices. We investigate the extremal properties of the graph invariant $\sigma$ over
the family $\mathcal{T}_n$ of all trees on $n \ge 3$ vertices. Specifically,
we compare the minimum values of $\sigma$ attained in restricted subclasses---including
caterpillar trees and greedy trees---with the global minimum over $\mathcal{T}_n$.
We prove that caterpillar trees do not achieve the minimum value of $\sigma$
among all trees, whereas greedy trees attain values no smaller than this global
minimum. Moreover, we show that certain trees, which are neither caterpillars
nor greedy trees, have $\sigma$-values strictly between the global minimum over
$\mathcal{T}_n$ and the minimum among caterpillar trees. These results highlight
structural limitations of these common tree classes in extremal problems and
offer new insights into the role of non-caterpillar, non-greedy trees in
minimizing graph invariants.

\end{abstract}

\noindent\rule{12.2cm}{1.0pt}

\noindent\textbf{AMS Classification 2010:} 05C05, 05C12, 05C20, 05C25, 05C35, 05C76, 68R10.


\noindent\rule{12.2cm}{1.0pt}

\section{Introduction}
Throughout this paper. Let $\GG=(V,E)$ be a graph with $|V(\GG)|=n$ vertices and $|E(\GG)|=m$ edges. 
For a vertex $v \in V(\GG)$, denote its degree by $d(v)$. 
A vertex of degree one is called a \emph{pendant vertex}. 
If $u$ and $v$ are adjacent vertices in $\GG$, the edge connecting them is denoted by $uv$.
Mathematicians and chemists have defined and explored various topological indices~\cite{gutman2018inverse,Todeschini2000,Todeschini2009}. 
One of the earliest and most influential measures of graph irregularity is the \emph{Albertson index}, introduced in~\cite{albertson1997irregularity}, and defined for a simple graph $\GG=(V,E)$ as:
$$
\mathrm{irr}(\GG)=\sum_{uv \in E(\GG)} |d_\GG(u)-d_\GG(v)|,
$$
This index captures local disparities between adjacent vertex degrees and has inspired numerous variants and generalizations~\cite{abdo2014total,abdo2019graph}. More recently, the \emph{Sigma index} was introduced as a quadratic analogue of the Albertson index~\cite{gutman2018inverse}, defined by:
$$
\sigma(\GG)=\sum_{uv \in E(\GG)} (d_\GG(u)-d_\GG(v))^2=F(\GG)-2M_2(\GG),
$$
where $F(\GG)$ is the sum of squared degrees over all edges, and $M_2(\GG)$ is the second Zagreb index given~\cite{gutman1972total, gutman1975acyclic} by 
\[
F(\GG)=\sum_{uv\in E(\GG)} (d_\GG(u)^2+d_\GG(v)^2) \quad M_2(\GG)=\sum_{uv\in E(\GG)} d_\GG(u)\,.\,d_\GG(v).
\]
This formulation amplifies large degree differences and provides a more sensitive measure of irregularity. Réti~\cite{Reti2019} compared the sigma index with several well-known irregularity measures and reported interesting properties of this index. Akbar~\cite{Akbar2023} determined the connected $k$-cyclic graphs of fixed order $n$ that maximize the $\sigma$-index, where a connected $k$-cyclic graph is a connected graph of order $n$ and size $n+k-1$.

 \section{Preliminaries}\label{sec2}
This section presents essential definitions and previously established results that form the theoretical foundation for our analysis of graph irregularity in trees.
\begin{definition}[Caterpillar tree]
A tree $T$ is called a \emph{caterpillar} if there exists a path $P$ in $T$ (called the \emph{spine}) such that every vertex of $T$ is either on $P$ or adjacent to exactly one vertex of $P$.
Equivalently, deleting all leaves (degree-$1$ vertices) of $T$ leaves a path graph (see Figure~\ref{f001Caterpillartree}).
\end{definition}
\begin{figure}[H]
    \centering
\begin{tikzpicture}[scale=.6]
\draw (1,3)-- (3,3);
\draw (3,4)-- (3,2);
\draw (3,3)-- (5,3);
\draw (4,4)-- (6,2);
\draw (6,4)-- (4,2);
\draw (5,3)-- (7,3);
\draw (7,3)-- (7,4);
\draw (7,3)-- (7,2);
\draw (7,3)-- (7.617593293911794,1.980333514937575);
\draw (7,3)-- (7.583838198655158,4.016890928754609);
\draw (7,3)-- (6.4,1.94);
\draw (7,3)-- (6.36,4);
\draw (7,3)-- (9,3);
\draw (9,3)-- (9,4);
\draw (9,3)-- (9,2);
\draw (9,3)-- (11,3);
\draw (11,3)-- (13,3);
\begin{scriptsize}
\draw [fill=black] (1,3) circle (1.5pt);
\draw [fill=black] (3,3) circle (1.5pt);
\draw [fill=black] (3,4) circle (1.5pt);
\draw [fill=black] (3,2) circle (1.5pt);
\draw [fill=black] (5,3) circle (1.5pt);
\draw [fill=black] (4,4) circle (1.5pt);
\draw [fill=black] (6,2) circle (1.5pt);
\draw [fill=black] (6,4) circle (1.5pt);
\draw [fill=black] (4,2) circle (1.5pt);
\draw [fill=black] (7,3) circle (1.5pt);
\draw [fill=black] (7,4) circle (1.5pt);
\draw [fill=black] (7,2) circle (1.5pt);
\draw [fill=black] (7.617593293911794,1.980333514937575) circle (1.5pt);
\draw [fill=black] (7.583838198655158,4.016890928754609) circle (1.5pt);
\draw [fill=black] (6.4,1.94) circle (1.5pt);
\draw [fill=black] (6.36,4) circle (1.5pt);
\draw [fill=black] (9,3) circle (1.5pt);
\draw [fill=black] (9,4) circle (1.5pt);
\draw [fill=black] (9,2) circle (1.5pt);
\draw [fill=black] (11,3) circle (1.5pt);
\draw [fill=black] (13,3) circle (1.5pt);
\end{scriptsize}
\end{tikzpicture}
    \caption{Caterpillar tree.}
    \label{f001Caterpillartree}
\end{figure}
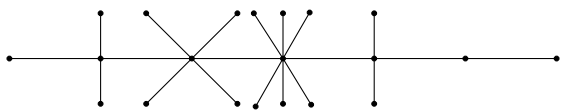
Several classical indices, such as the Zagreb indices, are also used to quantify structural features of graphs, particularly in chemical graph theory.

\begin{definition}[Greedy tree]
Let $\mathscr{D} = (d_1 \geqslant d_2 \geqslant \dots \geqslant d_n)$ be the graphic degree sequence of a tree. The \emph{greedy tree} $T_\mathscr{D}$ is constructed as follows:

\begin{enumerate}
  \item Start with a single vertex $v_1$ of target degree $d_1$.
  \item For $i=2,\dots,n$: Create a new vertex $v_i$ of target degree $d_i$. Connect $v_i$ to the vertex $u$ in the current tree that has the \emph{smallest current degree} among all vertices that have not yet reached their target degree (see Figure~\ref{fig001greedy}).
\end{enumerate}
\end{definition}

\begin{figure}[H]
    \centering
\centering
\begin{tikzpicture}[scale=1.2]
\draw (5,8)-- (3,7);
\draw (3,7)-- (3,6);
\draw (3,7)-- (2,6);
\draw (3,7)-- (4,6);
\draw (5,8)-- (5,7);
\draw (5,7)-- (4.5,6);
\draw (5,7)-- (5.5,6);
\draw (5,8)-- (6.5,6.5);
\draw (5,8)-- (8.5,6.5);
\draw (6.5,6.5)-- (6.5,5.5);
\draw (6.5,5.5)-- (6,5);
\draw (6.5,5.5)-- (6.5,5);
\draw (6.5,5.5)-- (7,5);
\draw (8.5,6.5)-- (8.5,5.5);
\draw (8.5,5.5)-- (9,5);
\draw (2,6)-- (1.7444795600926042,5.0005898689178885);
\draw (3,6)-- (2.5,5);
\draw (3,6)-- (3.5,5);
\draw (4,6)-- (4,5);
\draw (4.5,6)-- (4.5,5);
\draw (5.5,6)-- (5.5,5);
\draw (1.7444795600926042,5.0005898689178885)-- (1.46,4.04);
\draw (1.7444795600926042,5.0005898689178885)-- (2,4);
\draw (2.5,5)-- (2.5,4);
\draw (3.5,5)-- (3.5,4);
\draw (4,5)-- (4,4);
\draw (4.5,5)-- (4.5,4);
\draw (5.5,5)-- (5.5,4);
\draw (6,5)-- (6,4);
\draw (6.5,5)-- (6.5,4);
\draw (7,5)-- (7,4);
\draw (8.5,5.5)-- (8,5);
\draw (8,5)-- (8,4);
\draw (9,5)-- (9,4);
\begin{scriptsize}
\draw [fill=black] (5,8) circle (1.5pt);
\draw [fill=black] (3,7) circle (1.5pt);
\draw [fill=black] (3,6) circle (1.5pt);
\draw [fill=black] (2,6) circle (1.5pt);
\draw [fill=black] (4,6) circle (1.5pt);
\draw [fill=black] (5,7) circle (1.5pt);
\draw [fill=black] (4.5,6) circle (1.5pt);
\draw [fill=black] (5.5,6) circle (1.5pt);
\draw [fill=black] (6.5,6.5) circle (1.5pt);
\draw [fill=black] (8.5,6.5) circle (1.5pt);
\draw [fill=black] (6.5,5.5) circle (1.5pt);
\draw [fill=black] (6,5) circle (1.5pt);
\draw [fill=black] (6.5,5) circle (1.5pt);
\draw [fill=black] (7,5) circle (1.5pt);
\draw [fill=black] (8.5,5.5) circle (1.5pt);
\draw [fill=black] (9,5) circle (1.5pt);
\draw [fill=black] (1.7444795600926042,5.0005898689178885) circle (1.5pt);
\draw [fill=black] (2.5,5) circle (1.5pt);
\draw [fill=black] (3.5,5) circle (1.5pt);
\draw [fill=black] (4,5) circle (1.5pt);
\draw [fill=black] (4.5,5) circle (1.5pt);
\draw [fill=black] (5.5,5) circle (1.5pt);
\draw [fill=black] (1.46,4.04) circle (1.5pt);
\draw [fill=black] (2,4) circle (1.5pt);
\draw [fill=black] (2.5,4) circle (1.5pt);
\draw [fill=black] (3.5,4) circle (1.5pt);
\draw [fill=black] (4,4) circle (1.5pt);
\draw [fill=black] (4.5,4) circle (1.5pt);
\draw [fill=black] (5.5,4) circle (1.5pt);
\draw [fill=black] (6,4) circle (1.5pt);
\draw [fill=black] (6.5,4) circle (1.5pt);
\draw [fill=black] (7,4) circle (1.5pt);
\draw [fill=black] (8,5) circle (1.5pt);
\draw [fill=black] (8,4) circle (1.5pt);
\draw [fill=black] (9,4) circle (1.5pt);
\end{scriptsize}
\end{tikzpicture}
\caption{A greedy tree.}
    \label{fig001greedy}
\end{figure}
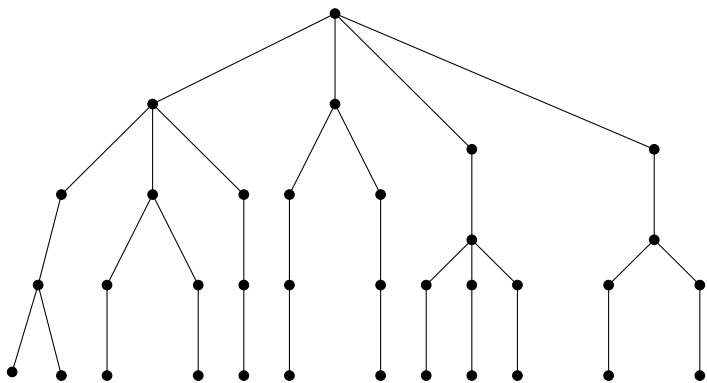

\begin{lemma}[\cite{gutman2018inverse}]
Let $ T $ be a tree with $ n \geq 3 $ vertices. Then:
\[
\sigma_{\max}(T) = (n - 1)(n - 2), \quad \sigma_{\min}(T) = 0.
\]
\end{lemma}

\begin{lemma}[\cite{yang2023sigma}]
Let $ G $ be a connected graph with $ n \geq 3 $ vertices and $ 1 \leq p \leq n - 3 $ pendant vertices. Then the maximum sigma index is attained when $ \Delta(G) = n - 1 $.
\end{lemma}

We conclude this section with a general formula for the Albertson index in trees with arbitrary degree sequences. This result provides the analytical foundation for our investigation of caterpillar structures.
\begin{theorem}[\cite{HamoudwithDuaa}]~\label{mainthm}
Let $\CC$ be a caterpillar tree with a degree sequence is $\mathscr{D}=(d_1,\dots,d_n),$ where $d_n\geqslant \dots \geqslant d_1$.
Then, Albertson index of tree $T$ is
    \[
    \mathrm{irr}(\CC)=d_1^2+d_n^2+\sum_{i=2}^{n-1} d_i^2+\sum_{i=2}^{n-1} d_i+d_n - d_1-2n+2.
    \]
\end{theorem}
We consider caterpillar trees denoted by $\CC(n, m)$, where $n$ is the number of backbone (or path) vertices and $m$ is the number of pendant vertices attached to each.  

\begin{theorem}[\cite{HamoudwithDuaa}]~\label{basic01sigma}
Let $\CC$ be a caterpillar tree with a degree sequence
$\mathscr{D}=(d_1,\dots,d_n),$ where $d_1\geqslant d_2\geqslant \dots
\geqslant d_n$. Then, Sigma index among caterpillar tree $\CC$ is
\[
\sigma(\CC)=(d_n-1)^3+(d_1-1)^3+\sum_{i=1}^{n}(d_i-d_{i+1})^2+\sum_{i=2}^{n-1}(d_i-1)^2(d_i-2).
\]
\end{theorem}
For a duplicate star graph $\mathcal{S}_{r,k}$, Theorem~\ref{thmpirlllkj} provide us the relationship of Sigma index with special terms.
\begin{theorem}[\cite{gutman2018inverse}]~\label{thmpirlllkj}
Let $\mathcal{S}_{r,k}$ be a duplicate star graph where $\deg_{\mathcal{S}_{r,k}}(u)=k$ and $d_{\mathcal{S}_{r,k}}(v)=r$. Then, 
\[
\sigma(\mathcal{S}_{r,k})=(k-1)^3+(r-1)^3+(k-r)^2.
\]
\end{theorem}

Let $\mathcal{T}_{n,p}$ be a tree of $n$ vertices and $p$ pendent vertices. Consider the caterpillar tree $\CC\in \mathcal{T}_{n,p}$. In this study we refer to caterpillar tree  $\CC$ by considering Figure~\ref{figTypCater01}.
\begin{figure}[H]
    \centering
\begin{tikzpicture}[scale=.8]
\draw   (1,2)-- (3,2);
\draw   (3,2)-- (5,2);
\draw   (7,2)-- (9,2);
\draw [line width=1pt,dotted] (5,2)-- (7,2);
\draw [line width=1pt,dotted] (3,2)-- (3,0);
\draw [line width=1pt,dotted] (5,2)-- (5,0);
\draw [line width=1pt,dotted] (7,2)-- (7,0);
\draw [line width=1pt,dotted] (9,2)-- (9,0);
\draw [line width=1pt,dotted] (1,2)-- (1,0);
\draw (0.8399938252334064,2.999767639650283) node[anchor=north west] {$d_1$};
\draw (2.816853283710542,2.974744102201205) node[anchor=north west] {$d_2$};
\draw (4.918830429433066,2.9246970273030497) node[anchor=north west] {$d_3$};
\draw (6.745548663215736,2.974744102201205) node[anchor=north west] {$d_{n-1}$};
\draw (8.722408121692872,3.0498147145484378) node[anchor=north west] {$d_n$};
\draw (9.097761183429038,0.6) node[anchor=north west] {$p$};
\draw (7.1709487998500565,0.6) node[anchor=north west] {$p$};
\draw (5.119018729025687,0.6) node[anchor=north west] {$p$};
\draw (3.1421592705485515,0.6) node[anchor=north west] {$p$};
\draw (0.2,0.6) node[anchor=north west] {$p$};
\begin{scriptsize}
\draw [fill=black] (1,2) circle (1.5pt);
\draw [fill=black] (3,2) circle (1.5pt);
\draw [fill=black] (5,2) circle (1.5pt);
\draw [fill=black] (7,2) circle (1.5pt);
\draw [fill=black] (9,2) circle (1.5pt);
\end{scriptsize}
\end{tikzpicture}
    \caption{Caterpillar tree $\CC(n,p)$.}
    \label{figTypCater01}
\end{figure}
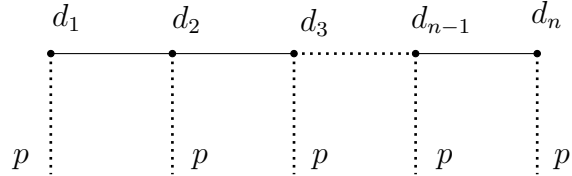
In next lemma, X.~Sun et al.~\cite{Sun2025DuJMei} had provided an important upper bound of Sigma index. It establishes the following section.
\begin{lemma}[\cite{Sun2025DuJMei}]
A tree $T$ of order $n \ge 5$ with $p$ pendent vertices, where $3 \leq p \leq n-2$, satisfies $\sigma(T) \leq (p-1)^{3} + (p-2)^{2} + 1.$
Equality holds if and only if $T \cong \mathcal{T}_{n,p}$.
\end{lemma}
This section has provided the essential theoretical tools—definitions, bounds, and structural properties—needed for our investigation of graph irregularity. In what follows, we apply these foundations to derive exact formulas and extremal results for  Sigma index in caterpillar trees.

\subsection{Problem of search}

\begin{problem}~\label{problemnumber1}
Consider a tree $\TT$ assuming $n \geq 3$ where
\begin{itemize}
\item Spine vertices: $v_1, \dots, v_n$, with degrees $p+1$ for $v_1$ and $v_n$, and $p+2$ for the internal spine vertices.
\item First-level pendants: each of degree $1+r$.
\item Second-level pendants: each of degree $1+s$.
\item Third-level pendants (leaves): each of degree $1$.
\end{itemize} 
Determine $\sigma$- irregularity index for the tree $\TT$.
\end{problem}

\section{Main Result}
In this section, consider a tree $\TT$ assuming $n \geq 3$. To solve the Problem~\ref{problemnumber1} we will discuss many cases based on the terms of the problem. Proposition~\ref{jas-pron1} had presented the concept of Sigma index among caterpillar tree by compare with other trees.

\begin{proposition}~\label{jas-pron1}
Let $\mathscr{D}=(d_1,d_2,\dots,d_n)$ be any tree degree sequence on $n$ vertices with $d_1\geqslant d_2\geqslant \dots\geqslant d_n\geqslant 1$. Among all trees on $n$ vertices realizing $\mathscr{D}$, the greedy tree $\TT_\mathscr{D}$ minimizes the $\sigma$-irregularity index:
\[
\sigma(\TT_\mathscr{D}) = \min \bigl\{ \sigma(T) \mid T \text{ is a tree with degree sequence } \mathscr{D} \bigr\}.
\]
\end{proposition}
\begin{proof}
Consider $\mathscr{D}=(d_1,d_2,\dots,d_n)$ be any tree degree sequence on $n$ vertices with $d_1\geqslant d_2\geqslant \dots\geqslant d_n\geqslant 1$. Assume $\TT$ be any tree with degree sequence $\mathscr{D}$ that is not the greedy tree $\TT_\mathscr{D}$. Let $v_1$ be a vertex of maximum degree $d_1$. By the monotone adjacency property, its neighbors are precisely the $d_1$ vertices with the next largest degrees. Otherwise, a vertex of larger degree would be adjacent to $v_1$ while a vertex of smaller degree remains unconnected.
Then, we find that if $\TT$ contains two edges $uv$ and $xy$ (with $u \neq x$) such that $d_\TT(u)>d_\TT(x)$, both $u$ and $x$ had remaining degree capacity when $y$ was attached to $x$, and $d_\TT(v)<d_\TT(y)$.
Let $\TT'$ be a tree derived from $\TT$ by remove both of edges $uv$ and $xy$, and add the edges $uy$ and $xv$. Thus, if there exist vertices $u,v,x,y$ such that $d(u) \geq d(v) \geq d(x) \geq d(y)$,
with edges $uv, vx \in E(\TT)$ where $ux, vy \notin E(\TT)$. Then, 
\begin{align*}
 \sigma(\TT)-\sigma(\TT') & = \bigl[(d(u)-d(v))^2+(d(x)-d(y))^2\bigr]-\bigl[(d(u)-d(y))^2+(d(x)-d(v))^2\bigr]\\
 &=-2\,(d(u)-d(x))\,(d(v)-d(y))
\end{align*}
Thus, in this case we find that $\sigma(\TT)>\sigma(\TT')$. This implies that in any $\sigma$-minimal tree, edges respect the degree ordering: higher-degree vertices cannot connect to disproportionately lower-degree vertices when a more balanced pairing is possible. 
Thus, for any two vertices $u,v$ with $d(u) \geq d(v)$, the degrees of the neighbors of $u$ are at least as large as those of the neighbors of $v$, up to permutation.
Therefore, every non-greedy tree can be transformed into the greedy tree by a finite sequence of strictly $\sigma$-decreasing operations. It follows that $\sigma(T)>\sigma(T_\mathscr{D})$. 
This completes the proof.
\end{proof}

The preceding observations highlight the importance of the tree's structural type in determining the value of the sigma index. To elucidate this further, Proposition~\ref{TmuDisu01min} provides a detailed analysis of the sigma index for trees that are neither caterpillars nor terraced trees.

\begin{proposition}~\label{TmuDisu01min}
Let $\mathcal{T}_n$ be the family of all trees on $n$ vertices where $n \geq 3$.
Among all trees on $n$ vertices that are neither caterpillars nor greedy trees. Then, 
\begin{align*}
&\min\{\sigma(\TT)\mid \TT\in \mathcal{T}_n; \text{ if } \TT \text{ is caterpillar tree}\}>\min\{\sigma(\TT)\mid \TT\in \mathcal{T}_n\}\\
&\min\{\sigma(\TT)\mid \TT\in \mathcal{T}_n; \text{ if } \TT \text{ is greedy tree}\}\geqslant \min\{\sigma(\TT)\mid \TT\in \mathcal{T}_n\}.
\end{align*}
 There exist non-caterpillar, non-greedy trees $\TT$ on $n$ vertices such that
 \begin{equation}~\label{eqq1TmuDisu01min}
\min\bigl\{\sigma(\TS)\mid \TS\in\mathcal{T}_n\bigr\}<\sigma(\TT)<\min\bigl\{\sigma(\TS)\mid \TS\in\mathcal{T}_n\bigr\} 
 \end{equation}
 where $\TS$ is a caterpillar tree.
\end{proposition}
\begin{proof}
Assume $\mathcal{T}_n$ be the family of all trees on $n$ vertices where $n \geq 3$. If $\TT$ is a caterpillar tree. Since $\mathcal{T}_n$ is finite, the quantity $\mu=\min\{\sigma(\TT)\mid \TT\in\mathcal{T}_n\}$ is attained by at least one tree. Then, the class of caterpillar trees is a proper subset of $\mathcal{T}_n$ for $n\geqslant 4$. Thus, 
\begin{equation}~\label{eqq2TmuDisu01min}
\mu \leqslant \min\{\sigma(\TT)\mid \TT\in\mathcal{T}_n\},
\end{equation}
the inequality is strict. Similarity,  for greedy trees are constructed to be locally optimal with respect to $\sigma$. Consequently, the value of $\sigma$ attained by any greedy tree cannot be smaller than the global minimum of $\sigma$ over $\mathcal{T}_n$.

Assume $\mu_c=\min\{\sigma(\TT) \mid \TT \in \mathcal{T}_n\},$
where $\TT$ is caterpillar tree, according to~\eqref{eqq2TmuDisu01min} we find that $\mu<\mu_c$.
As $\sigma$ takes only finitely many distinct values on $\mathcal{T}_n$, there exists $\TT^\ast \in \mathcal{T}_n$ such that
\begin{equation}~\label{eqq3TmuDisu01min}
\mu<\sigma(T^\ast)<\mu_c.
\end{equation}

Such a $\TT^\ast$ is necessarily non-caterpillar.
This tree is neither a caterpillar (its spine fails to dominate due to deep branches) nor greedy (it avoids balanced or star-like growth from low-degree vertices). Its value of $\sigma(T)$ exceeds that of typical double-star-like trees because of a more uniform degree distribution, yet remains strictly smaller than that of most caterpillar trees, whose long spines produce many substantial terms of the form $(d(u)-d(v))^2$.

Consequently, there exists a non-empty open interval that separates the global minimum of $\sigma$ from the minimum value of $\sigma$ attained over the class of caterpillar trees. Non-caterpillar, non-greedy trees achieve $\sigma$-values lying strictly inside this interval. Therefore, relations \eqref{eqq2TmuDisu01min} and \eqref{eqq1TmuDisu01min} together imply \eqref{eqq1TmuDisu01min}, which completes the proof.
\end{proof}

\begin{proposition}~\label{pro.sig.1}
The Sigma index of the caterpillar tree $\CC(n,m)$ of order $(n,m)$ is given by:
\[
\sigma(\CC(n,m)) = \begin{cases}
2m^3, & \quad \text{if } n = 2, \\
2m^3 + m\,.n\,(m+1)^2+2, & \quad \text{if } n> 2.
\end{cases}
\]
\end{proposition}

\begin{proof}
Assume the sequence $x_1, x_2, \dots, x_n$ and the pendant vertices $y_{1,1}, \dots, y_{i,j}$ are adjacent to $x_i$. Then $d(x_1) = d(x_n) = m+1$ and $d(x_i)=m+2$ for $2 \leqslant i \leqslant n-1$. Hence, 
\[
\operatorname{irr}(\CC(n,m)) = \begin{cases} 
m(m+1)n - 2m + 2, & \quad \text{if } n \geqslant 3, \\
m(m+1)n - 2m, & \quad \text{if } n \in \{1,2\}.
\end{cases}
\]

For $n = 2$, the Sigma index of the caterpillar tree is given by: $\sigma = 2m^3$.  
For $n = 3$, we have: $\sigma(\CC(n,m)) = 2m^3 + m(m+1)^2+2$, where $d(x_1)=d(x_n)=m+1$ at both ends of $\CC(n,m)$ and $d(x_2)=m+2$ in the central vertex of $\CC(n,m)$.  
Therefore, 
\[
\sigma(\CC(n,m)) = \begin{cases}
2m^3, & \quad \text{if } n = 2, \\
2m^3 + m\,.n\,(m+1)^2+2, & \quad \text{if } n> 2.
\end{cases}
\]

The Sigma index is then computed as $(d(x_1) - d (x_2))^2 + \dots + (d (x_{n+1}) - d (x_n))^2$, which can be expressed as:
\[
\sigma(\CC(n,m)) = \sum_{1=1}^n\left( \sum_{j=1}^n d(x_i) - d (y_{i,j}) \right)^2+\left( \sum_{i=1}^{n-1} d(x_i) - d (x_{i+1}) \right)^2.
\]
\end{proof}

Lemma~\ref{le.segma3} establishes the Sigma index of $\CC$ with degree sequence $\mathscr{D} = (d_1, d_2, d_3)$, under the assumption that $d_3 \geqslant d_2 \geqslant d_1$.

\begin{lemma}\label{le.segma3}
Let $\CC$ be a caterpillar tree with degree sequence $ \mathscr{D} = (d_1, d_2, d_3) $, where $ d_3 \geqslant d_2 \geqslant d_1 $. Then, the Sigma index of $ T $ is given by:
\begin{equation}~\label{eqle.segma3n1}
\sigma(\CC)=\sum_{i=1}^{3} p\,(d_i-1)^2+\sum_{i=1}^{2}(d_i-d_{i+1})^2.
\end{equation}
\end{lemma}
\begin{proof}
Let $\CC$ be a caterpillar tree with an increasing degree sequence $\mathscr{D} = (d_1, d_2, d_3)$ and pendent vertices $p$. To prove \eqref{eqle.segma3n1}, we consider the following valid cases. Suppose the degree sequence is $\mathscr{D}$ ordered such that $(d_1, d_2, d_3)$. Then, we have:
\begin{equation}~\label{eqq1le.segma3}
 \sigma(\CC)=p(d_1-1)^2+p(d_2-1)^2+p(d_3-1)^2+(d_1-d_2)^2+(d_2-d_3)^2  
\end{equation}
If $d_1=d_2=d_3=k$. Then, $\sigma(\CC)=3p\,(k-1)^2$. Now, consider the degree sequence $\mathscr{D}$ ordered such that $(d_2, d_1, d_3)$. We obtain:
\begin{equation}~\label{eqq2le.segma3}
 \sigma(\CC)=p(d_1-1)^2+p(d_2-1)^2+p(d_3-1)^2+(d_2-d_1)^2+(d_1-d_3)^2   
\end{equation}
If $d_1=d_2=d_3=k$. Then, $\sigma(\CC)=3p\,(k-1)^2$. Thus, if the degree sequence $\mathscr{D}$ ordered such that$(d_1, d_3, d_2)$. We obtain:
\begin{equation}~\label{eqq3le.segma3}
 \sigma(\CC)=p(d_1-1)^2+p(d_2-1)^2+p(d_3-1)^2+(d_1-d_3)^2+(d_3-d_2)^2  
\end{equation}
Thus, from~\eqref{eqq2le.segma3}--\eqref{eqq3le.segma3} we find that 
\[
\sigma(\CC)=\begin{cases}
\sigma_{\max}(\CC)=p(d_1-1)^2+p(d_2-1)^2+p(d_3-1)^2+(d_1-d_3)^2+(d_3-d_2)^2 , \\
\sigma_{\min}(\CC)=p(d_1-1)^2+p(d_2-1)^2+p(d_3-1)^2+(d_1-d_2)^2+(d_2-d_3)^2.
\end{cases}
\]
This completes the proof.
\end{proof}
Lemma~\ref{TypeLemman1s} represents the first step in solving Problem~\ref{problemnumber1}, as we use it to demonstrate the interconnection between the three levels in the tree in order to find the value of the sigma index, which is essential for discussing the other cases.

\begin{lemma}~\label{TypeLemman1s}
Consider $\TT$ be a tree given by Problem~\ref{problemnumber1} with spine vertices $v_1,v_2,\dots,v_n$ where $n\geqslant 3$. Then, the $\sigma$-irregularity index of $\TT$ is
\begin{equation}~\label{eqq1TypeLemman1s}
\sigma(\TT)=nprs^3+p(n-2)(p+1-r)^2+2p(p-r)^2+npr(r-s)^2+2.
\end{equation}
\end{lemma}
\begin{proof}
Assume $\TT$ be a tree with a spine, vertices on three levels, such that the first-level pendants each have degree $1+r$, the second-level pendants each have degree $1+s$, and the third-level pendants (leaves) each have degree $1$. For the spine edges, we find that there are $n-1$ edges. Thus, for  $i=1$ and $i=n-1$ the degrees are $p+1$ and $p+2$ implies that $(p+1-(p+2))^2=1$. For $i=2,\dots,n-2$ both endpoints have degree $p+2$ implies that $(p+2-(p+2))^2=0$. Now, to determine the value of edges between spine vertices and first-level pendants, we noticed that the edges incident to $v_1$ or $v_n$ (there are $2p$ of them), the degrees $p+1$ and $1+r$ implies that $(p+1-(1+r))^2=(p-r)^2$ yields $2\,(p-r)^2$. Also, for edges incident to internal spine vertices $v_2,\dots,v_{n-1}$ (there are $(n-2)p$ of them), the degrees $p+2$ and $1+r$ implies that $(p+2-(1+r))^2=(p+1-r)^2$ yields $2\,(p+1-r)^2$. Thus, 
\begin{equation}~\label{eqq2TypeLemman1s}
\sigma(\TT_1)=2\,(p-r)^2+2\,(p+1-r)^2+2.
\end{equation}
For other levels, there are $npr$ such edges for the second level. Thus, both endpoints have degrees $1+r$ and $1+s$ implies that $((1+r)-(1+s))^2=(r-s)^2$ yields $npr(r-s)^2.$ To determine the edges between second-level pendants and leaves, we noticed that there are $nprs$ such edges. Thus, for the degrees are $1+s$ and $1$ implies that $((1+s)-1)^2=s^2,$ yields $nprs^3.$ Thus, 
\begin{equation}~\label{eqq3TypeLemman1s}
\sigma(\TT_2)=npr(r-s)^2+nprs^3.
\end{equation}
Therefore, from~\eqref{eqq2TypeLemman1s} and \eqref{eqq3TypeLemman1s} we obtain on $\sigma(\TT)=\sigma(\TT_1)+\sigma(\TT_2).$ Thus, the relationship~\eqref{eqq1TypeLemman1s} holds. This completes the proof.
\end{proof}

It is noteworthy that our previous results addressed trees with three levels under certain conditions. 
To strengthen these findings, we modify the conditions in Lemma~\ref{TypeLemman2s}, which elucidates how the indicator's value varies with these changes.

\begin{lemma}~\label{TypeLemman2s}
Consider $\TT$ be a tree with spine vertices $v_1,v_2,\dots,v_n$ where $n\geqslant 3$ and the following degree structure based on Problem~\ref{problemnumber1}. Each first-level pendant vertex has degree $(1+r)^2$. Each second-level pendant vertex has degree $(1+s)^2$ and is adjacent to exactly $(1+s)^2 - 1$ leaves (third-level pendants), each of degree $1$. Then, the $\sigma$-irregularity index of $\TT$ is
\begin{equation}~\label{eqq1TypeLemman2s}
\sigma(\TT)=\mu_0-(1+s)^2)^2+np\,((1+r)^2-1)\,((1+s)^2-1)^3+\mu_1+2,
\end{equation}
where $\mu_0=np\,((1+r)^2-1)\,((1+r)^2$ and $\mu_1=2p\,(p-(1+r)^2+1)^2+p(n-2)\,(p+1-(1+r)^2)^2$.
\end{lemma}
\begin{proof}
Assume $\TT$ be a tree with a spine, vertices on three levels, according to Lemma~\ref{TypeLemman1s}, the first-level pendants each have degree $(1+r)^2$, the second-level pendants each have degree $(1+s)^2$, and the third-level pendants (leaves) each have degree $1$. Thus, 
\begin{equation}~\label{eqq2TypeLemman2s}
\sigma(\TT_1)=2p\,(p-(1+r)^2+1)^2+p(n-2)\,(p+1-(1+r)^2)^2.
\end{equation}
Similarly, for the first and second level, each of the $np$ first-level pendants has degree $(1+r)^2$, one edge goes to the spine, so it has $(1+r)^2- 1$ children implies that $np((1+r)^2-1)$ and  each such edge connects a vertex of degree $(1+r)^2$ to a vertex of degree $(1+s)^2$ implies that $((1+r)^2-(1+s)^2)^2$ yields $np\,((1+r)^2-1)\,((1+r)^2-(1+s)^2)^2.$ Also, to determine the value of the second level to leaf edges, we find that each of the $np((1+r)^2-1)$ second-level pendants has degree $(1+s)^2$, one edge to its parent (first-level), so it has $(1+s)^2-1$ leaves. Thus there are
$np\,((1+r)^2-1)\,((1+s)^2-1)$ such edges. Therefore, we have 
\begin{equation}~\label{eqq3TypeLemman2s}
\sigma(\TT_2)=np\,((1+r)^2-1)\,((1+s)^2-1)^3.
\end{equation}
Hence, from~\eqref{eqq2TypeLemman2s} and \eqref{eqq3TypeLemman2s} we obtain on $\sigma(\TT)=\sigma(\TT_1)+\sigma(\TT_2).$ Thus, the relationship~\eqref{eqq1TypeLemman2s} holds. This completes the proof.
\end{proof}
To strengthen these results as a solid foundation for future outcomes, we present in Table~\ref{tab001ComparisonVal} a comparative illustration between the values of Lemma~\ref{TypeLemman1s}, represented by \eqref{eqq1TypeLemman1s}, and Lemma~\ref{TypeLemman2s}, represented by \eqref{eqq1TypeLemman2s}. 
We clearly observe a substantial increase in the values of the Sigma index through \eqref{eqq1TypeLemman2s}, which indicates that the sharp bounds we can apply to the Sigma index via \eqref{eqq1TypeLemman1s} do not apply to it via \eqref{eqq1TypeLemman2s}, as the conditions provided differ in both.

\begin{table}[H]
    \centering
\begin{tabular}{|c|c|c|c||c|c|c|c|}
\hline
p & $\sigma(T)$ & $\sigma(T_1)$ & $\sigma(T_1)-\sigma(T)$ & p & $\sigma(T)$ & $\sigma(T_1)$ & $\sigma(T_1)-\sigma(T)$ \\ \hline
3 & 15128 & 30900007 & 30884879 &4 & 43410 & 154915338 & 154871928 \\ \hline
5 & 103212 & 600318883 & 600215671&6 & 215474 & 1935790462 & 1935574988 \\ \hline
7 & 408816 & 1139523935 & 1139115119 &8 & 720738 & 802298194 & 801577456 \\ \hline
9 & 1198820 & 1519059771 & 1517860951 &10 & 1901922 & -910339258 & -912241180 \\ \hline
11 & 2901384 & -406266985 & -409168369 &12 & 4282226 & 1018071450 & 1013789224 \\ \hline
\end{tabular}
\caption{Comparison of value among Lemma~\ref{TypeLemman1s} and \ref{TypeLemman2s} with $n=10$.}
\label{tab001ComparisonVal}
\end{table}

Through Lemma~\ref{TypeLemman3s}, we confirm the results discussed via both Lemmas~\ref{TypeLemman1s} and \ref{TypeLemman2s}. Accordingly, we introduce new conditions based on Problem \ref{problemnumber1}, which yield distinct outcomes. This leads us to induct on the values of the sigma index to reach the general case. Here, we must not overlook the importance of extending the empowerment of these results to extremal trees, as well as for computing the extremal values of the sigma index.

\begin{lemma}~\label{TypeLemman3s}
Let $\TT$ be a tree with spine vertices $v_1,v_2,\dots,v_n$ where $n\geqslant 3$ and the following degree structure based on Problem~\ref{problemnumber1}. Each first-level pendant vertex has degree $2p$. Each second-level pendant vertex has degree $2p^2$ and is adjacent to exactly $2p^2-1$ leaves (third-level pendants), each of degree $1$. Then, the $\sigma$-irregularity index of $\TT$ is
\begin{equation}~\label{eqq1TypeLemman3s}
\sigma(\TT)=np(2p-1)(2p^2-1)^3+2p(1-p)^2+p(n-2)(2-p)^2+\mu+2,
\end{equation}
where $\mu=np(2p-1)(2p-2p^2)^2$.
\end{lemma}
\begin{proof}
Assume  $\TT$ be a tree with spine vertices $v_1,v_2,\dots,v_n$ where $n\geqslant 3$. For the spine vertices and to the first level according to Lemma~\ref{TypeLemman1s} we find that 
\begin{equation}~\label{eqq2TypeLemman3s}
\sigma(\TT_1)=2p\,(1-p)^2+p\,(n-2)(2-p)^2.
\end{equation}
Then, at the second level, each of the $np$ first-level pendant vertices is adjacent to exactly $2p-1$ second-level pendant vertices. Thus, there are $np(2p-1)$ edges connecting the first level to the second level. Each such edge joins a vertex of degree $2p$ to a vertex of degree $2p^2$, contributing $(2p^2 - 2p)^2 = 4p^2(p-1)^2$ to the $\sigma$-irregularity index. Therefore, the total contribution from these edges is
\begin{equation} \label{eqq3TypeLemman3s}
\sigma(\TT_2)=np(2p-1)(2p^2-2p)^2.
\end{equation}
Now consider the edges incident to the third level. Each second-level pendant vertex has degree $2p^2$ and is adjacent to exactly $2p^2-1$ leaves of degree 1. Hence, there are $np(2p-1)(2p^2-1)$ such edges, each contributing $(2p^2-1)^2$ to the $\sigma$-irregularity index. The total contribution from these edges is
\begin{equation} \label{eqq4TypeLemman3s}
\sigma(\TT_3)=np\,(2p-1)(2p^2-1)^3.
\end{equation}

Finally, combining the contributions from the spine and first level in \eqref{eqq2TypeLemman3s} with those from \eqref{eqq3TypeLemman3s} and \eqref{eqq4TypeLemman3s} yields
\[
\sigma(\TT)=\sigma(\TT_1)+\sigma(\TT_2)+\sigma(\TT_3),
\]
which gives~\eqref{eqq1TypeLemman3s} and completes the proof.
\end{proof}

Actually, the $\sigma$-irregularity index of a $k$-level balanced pendant tree had extended from Problem~\ref{problemnumber1} for new Problem~\ref{problemThnumber2}. 
We now extend the result to Theorem~\ref{le.sigma2}, 
which addresses the Sigma index of a tree $\TT$ with degree sequence $\mathscr{D}=(d_1, d_2, d_3, d_4)$, where $d_4 \geqslant d_3 \geqslant d_2 \geqslant d_1$.

\begin{theorem}~\label{le.sigma2}
Let $\CC$ be a caterpillar tree of order $n$, and let $\mathscr{D} = (d_1, d_2, d_3, d_4)$ be a degree sequence, where $d_4 \geqslant d_3 \geqslant d_2 \geqslant d_1$. Then, 
\[
\sigma(\CC)=\begin{cases}
\sigma_{\max}(\CC)=p \sum_{i=1}^{4} (d_i-1)^2+(d_4-d_1)^2+(d_1-d_3)^2+(d_3-d_2)^2, \\
 \sigma_{\min}(\CC)= p \sum_{i=1}^{4} (d_i-1)^2+(d_1-d_2)^2+(d_2-d_3)^2+(d_3-d_4)^2.
\end{cases}
\]
\end{theorem}

\begin{proof}
Let $\CC$ be a caterpillar tree of order $n$, and let $\mathscr{D} = (d_1, d_2, d_3, d_4)$ be a non-increasing degree sequence. We aim to determine which vertices at the ends of the sequence represent the largest values, and which vertices in the middle represent the smallest values.  Assume the degrees satisfy $a \leq b \leq c \leq d$, where $c$ and $d$ are the two largest. For any ordering $(x_1,x_2,x_3,x_4)$, the Sigma index is
$$
\sigma(\CC) = p\sum_{i=1}^{4}(x_i-1)^2 + \sum_{i=1}^{3}(x_i-x_{i+1})^2.
$$
The term $p\sum_{i=1}^{4}(x_i-1)^2$ depends only on the multiset $\{a,b,c,d\}$ and remains constant across all permutations. The condition that the two largest degrees are adjacent requires $(c,d)$ or $(d,c)$ to appear consecutively in $(x_1,x_2,x_3,x_4)$. Thus, we noticed that The six corresponding orderings are
$(c,d,a,b),\ (d,c,a,b),\ (a,c,d,b),$ and $(a,d,c,b),\ (a,b,c,d),\ (a,b,d,c).$ For each ordering, according to Lemma~\ref{le.segma3} we noticed that
\begin{equation}~\label{eqq1le.sigma2}
\sigma(\CC)= p\big[(a-1)^2+(b-1)^2+(c-1)^2+(d-1)^2\big] + \mathcal{S},
\end{equation}

where $\mathcal{S}$ denotes the sum of squared consecutive differences:
\begin{align*}
 &\mathcal{S}_{(a,b,c,d)}= (a-b)^2 + (b-c)^2 + (c-d)^2, \quad 
\mathcal{S}_{(a,b,d,c)}= (a-b)^2 + (b-d)^2 + (d-c)^2,\\
&\mathcal{S}_{(a,c,d,b)}= (a-c)^2 + (c-d)^2 + (d-b)^2,\quad
\mathcal{S}_{(a,d,c,b)}= (a-d)^2 + (d-c)^2 + (c-b)^2,\\
&\mathcal{S}_{(c,d,a,b)}= (c-d)^2 + (d-a)^2 + (a-b)^2,\quad
\mathcal{S}_{(d,c,a,b)}= (d-c)^2 + (c-a)^2 + (a-b)^2.
\end{align*}
In each case, the term involving the two largest degrees is either $(c-d)^2$ or $(d-c)^2$, appearing exactly once since $c$ and $d$ are adjacent only once in a sequence of length $4$.
Thus, the maximum value of Sigma index according to~\eqref{eqq1le.sigma2} is 
\begin{equation}
\sigma_{\max}(\CC)=p \sum_{i=1}^{4} (x_i - 1)^2 + (d-a)^2 + (a-c)^2 + (c-b)^2.
\end{equation}
Adjacent terms are closest, minimizing the sum of squared differences. Therefore, the minimum value  is 
\begin{equation}~\label{eqq3le.sigma2}
 \sigma_{\min}(\CC)= p \sum_{i=1}^{4} (x_i - 1)^2 + (a-b)^2 + (b-c)^2 + (c-d)^2.
\end{equation}
Thus, from~\eqref{eqq1le.sigma2}--\eqref{eqq3le.sigma2} holds.
This completes the proof.
\end{proof}
\begin{problem}~\label{problemThnumber2}
Let $\TT$ be a tree with $n \geq 3$ vertices defined by a spine $v_1, v_2, \dots, v_n$ and exactly $k$ levels of pendants with the following degree conditions, for the spine vertices $d(v_1)=d(v_n)=p+1$, $d(v_i)=p+2$ for $i=2,\dots,n-1$ ($1 \leq p \in \mathbb{Z}$), each spine vertex has exactly $p$ children of \textbf{level 1} (first-level pendants), for each $\ell = 1, 2, \dots, k-1$, every vertex at level $\ell$ has degree $d_\ell$ and exactly $d_\ell - 1$ children at level $\ell+1$ and All vertices at level $k$ are leaves, i.e., $d(v) = 1$ for every $v$ at level $k$.   Let $d_0$ degree of the spine vertex to which the level-1 pendant is attached. Then $d_0 = p+1$ for branches attached to $v_1$ or $v_n$ and $d_0 = p+2$ for branches attached to $v_2, \dots, v_{n-1}$.
Define $m_0=p$ (number of level-1 children per spine vertex) and $m_\ell = d_\ell - 1$ for $\ell = 1, 2, \dots, k-1$
(number of pendent vertices each level-$\ell$ vertex has toward level $\ell+1$).  Find the relationship of Sigma index among $\TT$.
\end{problem}

The solution to this problem is provided by Theorem~\ref{maThmn1cater}, which accounts for the two main classes of vertices---spine vertices and pendent vertices---together with their respective degrees. This fundamental result, derived from the arguments in Lemmas~\ref{TypeLemman2s} and \ref{TypeLemman3s}, addresses several cases that will be discussed later in connection with the setting introduced in Problem~\ref{problemThnumber2}.

\begin{theorem}~\label{maThmn1cater}
Let $\TT$ be a tree with $n \geq 3$ vertices given by Problem~\ref{problemThnumber2}, consisting of a spine $v_1, v_2, \dots, v_n$ and exactly $k$ levels of pendants. The $\sigma$-irregularity index is
\begin{equation}~\label{eqq1maThmn1cater}
\sigma(T)=\sum_{\ell=1}^{k-1} np\Bigl(\prod_{j=1}^{\ell-1}(d_j-1)\Bigr)(d_\ell-d_{\ell+1})^2+np\Bigl(\prod_{j=1}^{k-2}(d_j-1)\Bigr)(d_{k-1}-1)^2+\mu,
\end{equation}
where $\mu=2+\,2p(p+1-d_1)^2+p(n-2)(p+2-d_1)^2$.
\end{theorem}
\begin{proof}
The vertices of $\TT$ are partitioned into levels such that level $0$ consists of the spine vertices, while level $\ell$ satisfying $1\leqslant\ell\leqslant k$ contains all vertices at distance $\ell$ from the spine. Edges exist only between consecutive levels. Hence, we consider each spine vertex has exactly $m_0=p$ neighbors in level $1$. Consequently, the total number of edges between levels $0$ and $1$ is $np$.
For $\ell \geqslant 1$, every vertex in level $\ell$ has exactly $m_\ell=d_\ell-1$ pendent vertices in level $\ell+1$. Thus, the number of edges between levels $\ell$ and $\ell+1$ is
\begin{equation}~\label{eqq2maThmn1cater}
\mathcal{E}_{\ell,\ell+1}=np\,\prod_{j=1}^{\ell-1} (d_j-1).
\end{equation}
Therefore, from~\eqref{eqq2maThmn1cater} we noticed that
\begin{equation}~\label{eqq3maThmn1cater}
\sigma(\TT)=\sigma_{0,1}(\TT)+\sum_{\ell=1}^{k-1}\sigma_{\ell,\ell+1}(\TT),
\end{equation}
where $\sigma_{\ell,\ell+1}$ denotes the the terms arising from the edges connecting levels $\ell$ and $\ell+1$. For $\ell=1$, each spine vertex has exactly $p$ neighbors in level $1$. As there are $n$ spine vertices, the total number of edges between levels $1$ and $2$ is $np$. Each such edge connects a vertex of degree $d_1$ (in level $1$) to a vertex of degree $d_2$ (in level $2$). Then, $\sigma_{1,2}(\TT)=np\,(d_1-d_2)^2.$

Therefore, for $1 \leqslant \ell \leqslant k-2$, each edge connecting a vertex at level $\ell$ to a vertex at level $\ell+1$ contributes $(d_\ell-d_{\ell+1})^2$ to~\eqref{eqq3maThmn1cater}.  Therefore, according to~\eqref{eqq3maThmn1cater} we find that
\begin{equation}~\label{eqq4maThmn1cater}
\sigma_{\ell,\ell+1}(\TT)=np\,\Bigl(\prod_{j=1}^{\ell-1} (d_j-1)\Bigr)(d_\ell-d_{\ell+1})^2.
\end{equation}
Then, from~\eqref{eqq4maThmn1cater} by considering $\ell=k-1$, it satisfies the property that every vertex at level $k$ is a leaf of degree 1. Thus, 
\begin{equation}~\label{eqq5maThmn1cater}
\sigma_{k-1,k}(\TT)=np\,\Bigl(\prod_{j=1}^{k-2} (d_j-1)\Bigr)\,(d_{k-1}-1)^2.
\end{equation}

Now, for establishing $\sigma_{0,1}(\TT)$, we noticed that each edge between a spine vertex and a level-1 vertex contributes $(d_0-d_1)^2$, where $d_0$ denotes the degree of the spine vertex. Hence, two endpoint spine vertices $v_1$ and $v_n$ each have degree $d_0 = p+1$ and $p$ neighbors at level 1. Their combined contribution is therefore  
  $2p(p+1-d_1)^2$. Then, each of the $n-2$ internal spine vertices has degree $d_0 = p+2$ and $p$ neighbors at level 1 yields $p(n-2)(p+2-d_1)^2$. Additionally, the two spine edges $v_1v_2$ and $v_{n-1}v_n$ contribute $2$. Therefore,
\begin{equation}~\label{eqq6maThmn1cater}
\sigma_{0,1}(\TT)=2p(p+1-d_1)^2+p(n-2)(p+2-d_1)^2+2.
\end{equation}
Thus, in this case if we consider $\mu=\sigma_{0,1}(\TT)$ both of~\eqref{eqq5maThmn1cater} and \eqref{eqq6maThmn1cater} implies that the relationship~\eqref{eqq3maThmn1cater}. Therefore, the relationship~\eqref{eqq1maThmn1cater} holds. 
\end{proof}
Both extreme values play a pivotal role in the behavior of the Sigma index, as illustrated in Figure~\ref{fig001discussion}. The maximum value of Sigma index satisfied with the left figure and the minimum value satisfied with the right figure. Note that the extremal value is attained for the degree sequence ordered by $d_n>d_1>\dots>d_2>d_{n-1}$ and $d_n>d_2>\dots>d_{n-1}>d_1$.

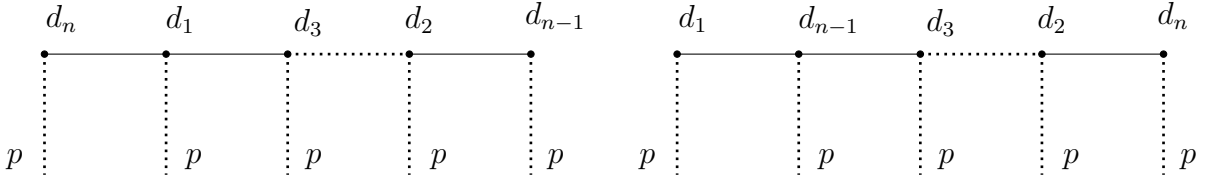
\begin{figure}[H]
    \centering
\begin{tabular}{cc}
\begin{tikzpicture}[scale=.8]
\draw   (1,2)-- (3,2);
\draw   (3,2)-- (5,2);
\draw   (7,2)-- (9,2);
\draw [line width=1pt,dotted] (5,2)-- (7,2);
\draw [line width=1pt,dotted] (3,2)-- (3,0);
\draw [line width=1pt,dotted] (5,2)-- (5,0);
\draw [line width=1pt,dotted] (7,2)-- (7,0);
\draw [line width=1pt,dotted] (9,2)-- (9,0);
\draw [line width=1pt,dotted] (1,2)-- (1,0);
\draw (0.8399938252334064,2.999767639650283) node[anchor=north west] {$d_n$};
\draw (2.816853283710542,2.974744102201205) node[anchor=north west] {$d_1$};
\draw (4.918830429433066,2.9246970273030497) node[anchor=north west] {$d_3$};
\draw (6.745548663215736,2.974744102201205) node[anchor=north west] {$d_2$};
\draw (8.722408121692872,3.0498147145484378) node[anchor=north west] {$d_{n-1}$};
\draw (9.097761183429038,0.6) node[anchor=north west] {$p$};
\draw (7.1709487998500565,0.6) node[anchor=north west] {$p$};
\draw (5.119018729025687,0.6) node[anchor=north west] {$p$};
\draw (3.1421592705485515,0.6) node[anchor=north west] {$p$};
\draw (0.2,0.6) node[anchor=north west] {$p$};
\begin{scriptsize}
\draw [fill=black] (1,2) circle (1.5pt);
\draw [fill=black] (3,2) circle (1.5pt);
\draw [fill=black] (5,2) circle (1.5pt);
\draw [fill=black] (7,2) circle (1.5pt);
\draw [fill=black] (9,2) circle (1.5pt);
\end{scriptsize}
\end{tikzpicture} & \begin{tikzpicture}[scale=.8]
\draw   (1,2)-- (3,2);
\draw   (3,2)-- (5,2);
\draw   (7,2)-- (9,2);
\draw [line width=1pt,dotted] (5,2)-- (7,2);
\draw [line width=1pt,dotted] (3,2)-- (3,0);
\draw [line width=1pt,dotted] (5,2)-- (5,0);
\draw [line width=1pt,dotted] (7,2)-- (7,0);
\draw [line width=1pt,dotted] (9,2)-- (9,0);
\draw [line width=1pt,dotted] (1,2)-- (1,0);
\draw (0.8399938252334064,2.999767639650283) node[anchor=north west] {$d_1$};
\draw (2.816853283710542,2.974744102201205) node[anchor=north west] {$d_{n-1}$};
\draw (4.918830429433066,2.9246970273030497) node[anchor=north west] {$d_3$};
\draw (6.745548663215736,2.974744102201205) node[anchor=north west] {$d_2$};
\draw (8.722408121692872,3.0498147145484378) node[anchor=north west] {$d_n$};
\draw (9.097761183429038,0.6) node[anchor=north west] {$p$};
\draw (7.1709487998500565,0.6) node[anchor=north west] {$p$};
\draw (5.119018729025687,0.6) node[anchor=north west] {$p$};
\draw (3.1421592705485515,0.6) node[anchor=north west] {$p$};
\draw (0.2,0.6) node[anchor=north west] {$p$};
\begin{scriptsize}
\draw [fill=black] (1,2) circle (1.5pt);
\draw [fill=black] (3,2) circle (1.5pt);
\draw [fill=black] (5,2) circle (1.5pt);
\draw [fill=black] (7,2) circle (1.5pt);
\draw [fill=black] (9,2) circle (1.5pt);
\end{scriptsize}
\end{tikzpicture}
    
\end{tabular}
    \caption{The behavior of caterpillar tree $\CC(n,p)$.}
    \label{fig001discussion}
\end{figure}
Theorem~\ref{TnpowPronu4} establishes the initial upper bounds by analyzing the sum of squared differences between the degrees in $\mathscr{D}$, while taking into account the effect of the average degree.

\begin{theorem}~\label{TnpowPronu4}
Let $\CC$ be a caterpillar tree, and let $\mathscr{D} = (d_1, d_2, \dots, d_n)$ be a degree sequence where $d_1 \geqslant d_2 \geqslant \dots \geqslant d_n$, with maximum degree $\Delta$, minimum degree $\delta \geqslant 2$, and let $\lambda_{\mathscr{D}}$ be the average of $\mathscr{D}$. Then, the upper bound of the Sigma index satisfies
\begin{equation}~\label{eq1TnpowPronu4}
\sigma(\CC)\leqslant \sum_{i=1}^{n-1}\lambda_{\mathscr{D}}(d_i-d_{i+1})^3 +2(n^2+m^2)+3m+n+2.
\end{equation}
\end{theorem}

\begin{proof}
Assume $\mathscr{D}=(d_1,d_2,\dots,d_n)$ is a degree sequence where $d_1 \geqslant d_2 \geqslant \dots \geqslant d_n$, with maximum degree $\Delta$, minimum degree $\delta \geqslant 2$, and average degree $\lambda_{\mathscr{D}}$. Clearly, $n$ satisfies 
$n > \sum_{i=1}^{n-1}(d_i - d_{i+1})^2$ and $2n > \sum_{i=1}^{n-1}(d_i - d_{i+1})^3.$ 
These $n$-related terms establish the lower bounds for the Sigma index. Then
\begin{equation}~\label{eq2TnpowPronu4}
\sigma(\CC) \geqslant \sum_{i=1}^{n-1}\lambda_{\mathscr{D}} (d_i-d_{i+1})^3+2n^2+2.
\end{equation}
We compare the term $2n^3+n^2-2n+1$ with the term $\Delta(\Delta-1)^2$ by considering the value of $m-1,$
\[
0 < \frac{1}{m-1}\left(\frac{2n^3 + n^2 - 2n + 1}{\Delta(\Delta - 1)^2}\right) < 3. 
\]
Based on this reduction, the lower bound~\eqref{eq2TnpowPronu4} is difficult to determine within a certain bound. Thus, we have $\Delta (n - \Delta)^2 - \Delta (m - \Delta)^2 \geqslant \sum_{i=1}^{n-1}\lambda_{\mathscr{D}} (d_i - d_{i+1})^3,$
and for the term $\Delta(\Delta - 1)^2$ compared with the term $(2m^2 + 2m + n) / 2(n + m)$, considering that $\sigma(\CC) < 2m^2 + 2m + n$ holds, we obtain
\[
(d_1 - d_n)^2 \geqslant \frac{1}{\Delta(\Delta - 1)^2} \left(\frac{2m^4 + 2m + n}{2(n + m)} \right).
\]
The sharp lower bound of the Sigma index satisfies the inequality by considering~\eqref{eq2TnpowPronu4} and the value of the term $(d_1 - d_n)^2$,
\begin{equation}~\label{eq3TnpowPronu4}
\sigma(\CC) > \frac{1}{\Delta(\Delta - 1)^2} \left(\frac{2m^4 + 2m + n}{2(n + m)} + 2nm + (d_1 - d_n)^2 \right).
\end{equation}
Thus, relationship~\eqref{eq3TnpowPronu4} implies that $\Delta (2n - m)^2$ satisfies the lower bound according to the value of $\lambda_{\mathscr{D}}$ and the upper bound according to the value $d_i^2$. Hence,
\begin{equation}~\label{eq4TnpowPronu4}
\sum_{i=1}^{n-1} \lambda_{\mathscr{D}} (d_i - d_{i+1})^3 \leqslant \frac{1}{4}\left( \frac{2m^4 + 2m + n}{2 \Delta (n + m)(2n - m)^2} + 2nm + (d_1 - d_n)^2 \right) \leqslant 3 \sum_{i=1}^n d_i^2 - (n + m),
\end{equation}
where
\[
2 \leqslant \frac{1}{4} \left( \frac{2m^4 + 2m + n}{2 \Delta (n + m)(2n - m)^2} \right) \leqslant 3.
\]
Therefore, from~\eqref{eq3TnpowPronu4} and \eqref{eq4TnpowPronu4}, the relationship with~\eqref{eq2TnpowPronu4} implies~\eqref{eq1TnpowPronu4}.
\end{proof}

\section{Conclusion}\label{sec5}
In this paper, we investigate a key problem concerning trees that differ from Problem~\ref{problemnumber1} and Problem~\ref{problemThnumber2}. We resolve Problem~\ref{problemnumber1}, rigorously proving that the sigma index satisfies the relation
\[
\sigma(\TT)=nprs^3+p(n-2)(p+1-r)^2+2p(p-r)^2+npr(r-s)^2+2.
\]
Similarly, we address Problem~\ref{problemThnumber2}, which establishes the result under the conditions specified 
\[
\sigma(T)=\sum_{\ell=1}^{k-1} np\Bigl(\prod_{j=1}^{\ell-1}(d_j-1)\Bigr)(d_\ell-d_{\ell+1})^2+np\Bigl(\prod_{j=1}^{k-2}(d_j-1)\Bigr)(d_{k-1}-1)^2+\mu.
\]
These findings lay a foundational cornerstone in graph theory, particularly for topological indices of trees and their extremal values. This work presents a unified analytical framework for comparing linear and quadratic irregularity measures in trees, featuring precise formulas validated on fixed-size examples and general sequences. The framework strengthens the theoretical foundations of irregularity indices and provides a reference for further studies of structured families, such as caterpillar trees.

\section*{Acknowledgements}
The authors would like to express their sincere gratitude to the anonymous reviewers for their insightful comments and constructive suggestions, which have substantially enhanced the clarity and quality of this manuscript.

\section*{Declarations}
\begin{itemize}
	\item Funding: Not Funding.
	\item Conflict of interest/Competing interests: The author declare that there are no conflicts of interest or competing interests related to this study.
	\item Ethics approval and consent to participate: The author contributed equally to this work.
	\item Data availability statement: All data is included within the manuscript.
\end{itemize}

\end{document}